\documentclass[12pt]{amsart}

\usepackage{latexsym,amsmath,amssymb,amsthm,amsfonts}

\usepackage[shortalphabetic]{amsrefs}

\newcommand{\ds}{\displaystyle}
\newcommand{\R}{{\mathbb{R}}}

\renewcommand{\P}{{\mathcal{P}}}

\newcommand{\x}{{\boldsymbol{x}}}

\renewcommand{\u}{{\boldsymbol{u}}}
\renewcommand{\v}{{\boldsymbol{v}}}
\newcommand{\w}{{\boldsymbol{w}}}

\newcommand{\y}{{\boldsymbol{y}}}
\newcommand{\z}{{\boldsymbol{z}}}
\newcommand{\zero}{{\boldsymbol{0}}}
\newcommand{\conv}{{\rm conv}}
\newcommand{\wt}{\widetilde}

\newcommand{\dist}{{\rm dist}}

\newtheorem{theorem}{Theorem}[section]
\newtheorem{lemma}[theorem]{Lemma}
\newtheorem{proposition}[theorem]{Proposition}
\newtheorem{conjecture}[theorem]{Conjecture}
\newtheorem{corollary}[theorem]{Corollary}

\theoremstyle{remark}
\newtheorem{remark}[theorem]{Remark}

\renewcommand{\phi}{\varphi}

\numberwithin{equation}{section}

\title{Upper estimates of Christoffel function on convex domains}

\author{A.\ Prymak}
\address{Department of Mathematics, University of Manitoba, Winnipeg, MB, R3T2N2, Canada}
\email{prymak@gmail.com}
\thanks{The author was supported by NSERC of Canada Discovery Grant RGPIN 04863-15.}



\begin{document}

\begin{abstract}
New upper bounds on the pointwise behaviour of Christoffel function on convex domains in $\R^d$ are obtained. These estimates are established by explicitly constructing the corresponding ``needle''-like algebraic polynomials having small integral norm on the domain, and are stated in terms of few easy-to-measure geometric characteristics of the location of the point of interest in the domain. Sharpness of the results is shown and examples of applications are given.
\end{abstract}

\maketitle

\section{Introduction, results and remarks}

For a compact set $D\subset\R^d$ with non-empty interior and a positive weight function $w\in L_1(D)$, the associated Christoffel function is defined as
\[
\lambda_n(D,w,\x)=\left(\sum_{k=1}^N\phi_k(\x)^2\right)^{-1}, \quad \x\in D,
\]
where $\P_n=\P_{n,d}$ is the space of all real algebraic polynomials of total degree $\le n$ in $d$ variables, and $\{\phi_k\}_{k=1}^N$ is an orthonormal basis of $\P_n$ with respect to the inner product $\langle f,g \rangle = \int_D f(\y)g(\y)w(\y)d\y$. Equivalently, the Christoffel function can be defined through the following extremal property:
\begin{equation}\label{def_lambda}
\lambda_n(D,w,\x)=\min_{f\in\P_n,\, |f(\x)|=1}\int_D f^2(\y)w(\y) d\y, \quad \x\in D.
\end{equation}
If $w\equiv 1$ is the uniform weight, we will write $\lambda_n(D,\x)=\lambda_n(D,w,\x)$ and this quantity will be of our primary interest. Christoffel functions play an extremely important role in the theory of orthogonal polynomials and other areas of analysis.

It was established in~\cite{Bo-elal} (see also~\cite{Xu} for related results) that for centrally symmetric positive continuous weight $w$ on the unit ball in $\R^d$ one has
\begin{equation}\label{ball_as}
\lim_{n\to\infty}{\lambda_n(B^d,w,\x)}{\binom{n+d}{d}}=\frac{\pi^{\frac{d+1}{2}}w(\x)\sqrt{1-|\x|^2}}{\Gamma(\frac{d+1}2)}, \quad |\x|<1,
\end{equation}
where $|\x|=(x_1^2+\dots+x_n^2)^{1/2}$ is the Euclidean norm of $\x=(x_1,\dots,x_d)$, and $B^d=\{\x\in\R^d:|\x|\le1\}$. This is an example of a typical result on computation of \emph{asymptotics} of Christoffel function, which usually establishes that
\[
\lim_{n\to\infty} n^{d} \lambda_n(D,w,\x) = \rho(D,w,\x)
\]
and explicitly computes the limit function $\rho(D,w,\x)$ at interior points $\x\in D$. When $d=1$ this has been done for quite general weights on the segment (see, e.g.~\cite{To}), while for higher dimensions only some special domains such as ball, simplex, cube have been covered for certain classes of weights.

One of the difficulties for the higher dimensions is understanding the influence of geometry of the domain on Christoffel function. In addition, there are usually no explicit expressions available for orthonormal polynomial bases on domains admitting any reasonable generality. In a recent important work~\cite{Kr} Kroo obtained sharp lower estimates on $\liminf_{n\to\infty} n^{d} \lambda_n(D,w,\x)$ for certain general classes of convex and star-like domains. One of the main motivations for the current work was the question whether an ``extra'' factor of $\log n$ could be removed in the sharpness result~\cite[Theorem~2]{Kr}, which will be answered affirmatively in Section~\ref{sect_examples}.

Rather than focusing on estimating the \emph{asymptotics} of Christoffel function, we will be concerned with its \emph{behavior}, i.e., computation of $\lambda_n(D,w,\x)$ up to a constant factor as a function of $n$ and $\x$. For example, it was established in~\cite[(7.14)]{Ma-To} that for doubling weights $w$ on $[-1,1]$ (in particular, for $w\equiv1$) and any $x\in[-1,1]$
\begin{equation}\label{eqn-varga}
\frac1c\int_{I_x}w(t)dt
\le \lambda_n([-1,1],w,x)
\le c\int_{I_x}w(t)dt,
\quad I_x=[x-\rho_n(x),x+\rho_n(x)]\cap[-1,1], 
\end{equation}
where $\rho_n(x)=n^{-2}+n^{-1}\sqrt{1-x^2}$ and $c$ depends on the doubling constant of $w$.
The estimates of behavior of Christoffel function are more useful in the sense that they allow to deduce bounds on asymptotics (up to a constant) and, for example, to compute the order of $\min_{\x\in D} \lambda_n(D,\x)$ as a function of $n$, which is not possible to imply from typical results on asymptotics. The quantity $\min_{\x\in D} \lambda_n(D,\x)$ is crucial for Nikol'skii inequalities (see~\cite{Di-Pr}) and has other applications, for instance, in analysis of least square approximation~\cite{Co-Da-Le}. 

For one dimension, as illustrated by~\eqref{eqn-varga}, the quantity $\rho_n(x)$ properly accounts for the boundary effect. In~\cite{Kr}, for the so-called $C^\alpha$ domains (where $\alpha$ reflects certain smoothness of the boundary), it is shown how the distance to the boundary (defined in terms of Minkowski functional) can be used for estimates of Christoffel function. In this paper, we work with general convex bodies (convex compact sets with non-empty interior) without any smoothness assumptions and show that apart from the distance to the boundary one can look at certain measurements of the size of an appropriate hyperplane section of the body to obtain precise upper bounds on Christoffel function.

In what follows, $c$, $c_0$, $c_1$, etc. denotes positive absolute constants, and  $c(\cdot)$ denotes positive constants depending only on parameters indicated in parentheses. These constants may be different at different occurrences even if the same notation is used. We write $A\approx B$ if $c_0 A \le B \le c_1 A$ for any values of variables that define the quantities $A$ and $B$. We always assume that $n$ is a positive integer. By $\partial D$ we denote the boundary of $D$ and also set $\dist(\x,Y):=\inf\{|\x-\y|:\y\in Y\}$.

Now let us state the main result for two dimensions.
\begin{theorem}\label{main_lower_thm}
Suppose a planar convex body $D$ is contained in a disc of radius $R$, and for some $\x\in D\setminus\partial D$ and unit vector $\u\in\R^2$ there are $r>0$ and $t_0<0$ such that $rB^2+\x+t_0\u\subset D$. Let $\delta:=\max\{t:\x+t\u\in D\}$ and $l_i:=\max\{t:\x+(-1)^it\v\in D\}$, $i=1,2$, where $\v$ is one of the two unit vectors orthogonal to $\u$. If $\delta\ge \sigma n^{-2}$, $\sigma>0$, then
\begin{equation}\label{main_lower_est}
\lambda_n(D,\x)\le c(r,R,\sigma) n^{-2}\sqrt{\min\{l_1l_2,\delta\}}.
\end{equation}
\end{theorem}

\begin{remark}\label{constructive}
Proof of Theorem~\ref{main_lower_thm} is constructive, i.e., following this proof and that of~\cite[Theorem~6.3]{Di-Pr}, one can explicitly construct the polynomials of degree $n$ with $P_n(\x)=1$ and $\|P_n\|_{L_2(D)}^2\le c(r) n^{-2}\sqrt{\min\{l_1l_2,\delta\}}$ (see~\eqref{def_lambda}).
\end{remark}

\begin{remark}
	The constant in~\eqref{main_lower_est} depends on $r$ as $r\to0+$ and does not depend on $t_0$. Alternatively, instead of requiring that $rB^2+\x+t_0\u\subset D$, one can define $r$ as follows:
	\[
	r=\sup_{\y\in D\cap\{\x+t\u:t<0\}} \dist(\y,\partial D).
	\] 
\end{remark}

Restriction $\delta\ge\sigma n^{-2}$ is not essential and was imposed only to simplify the statements of the results (in particular, to allow writing $n^{-1}\sqrt{\delta}$ rather than $n^{-2}+n^{-1}\sqrt{\delta}$). More precisely, the next proposition shows that one can always step towards inside the domain by an order of $n^{-2}$ leading to no change in the order of Christoffel function. We call $D\subset\R^d$ a star-like body in $\R^d$ (with respect to the origin), if $D$ is a compact set with non-empty interior and $tD\subset D$ for any $t\in[0,1]$.
\begin{proposition}\label{boundary}
If $D$ is a star-like body in $\R^d$, then for any point $\x\in D$
\[
\lambda_n(D, \x)\approx \lambda_n(D,\mu \x), \quad \mu\in[1-c(d)n^{-2},1],
\]
where $c(d)=2^{-3-d/2}$ (recall that the constants in the equivalence notation ``$\approx$'' are absolute).
\end{proposition}

The following theorem shows that the bound in Theorem~\ref{main_lower_thm} is sharp in the class of convex bodies if we only use measurements $\delta$ and $l_i$, $i=1,2$. Let us remark that under the conditions of Theorem~\ref{main_lower_thm} it is not hard to see that by convexity of $D$ we always have
\begin{equation}\label{li_delta}
\frac{r\delta}{\delta-t_0}B^2+\x\subset D, \quad\text{so}\quad
l_i\ge \frac{r\delta}{\delta-t_0}\ge \frac{r}{2R}\delta, \quad i=1,2.
\end{equation}

\begin{theorem}\label{upper_thm}
For any positive $l_1$, $l_2$, $\delta$ with $10 \delta < l_1, l_2 < \frac1{10}$, one can find a planar convex body $D$ and a point $\x\in D$ satisfying $B^2 \subset D\subset 4B^2$ and with $\u:=\x/|\x|$ that $\delta=\max\{t:\x+t\u\in D\}$ and $l_i=\max\{t:\x+(-1)^it\v\in D\}$, $i=1,2$, where $\v$ is one of the two unit vectors orthogonal to $\u$, and that for any $n$ with $\delta>\sigma n^{-2}$, $\sigma>0$, the following inequality holds:
\begin{equation}\label{upper_est}
\lambda_n(D,\x)\ge c(\sigma) n^{-2}\sqrt{\min\{l_1l_2,\delta\}}.
\end{equation}
\end{theorem}
Note that Theorem~\ref{main_lower_thm} is applicable for $D$ and $\x$ from Theorem~\ref{upper_thm} with $r=1$ and $R=4$.
\begin{remark}
The converse~\eqref{main_lower_est} of~\eqref{upper_est} is true for any convex body. We believe that the class of convex bodies for which~\eqref{upper_est} holds (for any $\x$ with $\delta>\sigma n^{-2}$) is rather wide, however, that it does not include all convex bodies. In other words, to compute the order of $\lambda_n(D,\x)$ for arbitrary convex $D$, one must use more measurements than only $\delta$ and $l_i$, $i=1,2$. Alternatively, one could restrict the class of considered bodies and impose some additional conditions apart from convexity.
\end{remark}

In $\R^2$, a hyperplane section of a planar convex body is a segment, and along with a point on this segment, such a configuration can be completely described by two parameters as we have done with $l_1$ and $l_2$ above. For $\R^d$, $d>2$, a hyperplane section of a convex body in $\R^d$ is a convex body in $\R^{d-1}$, which makes things much more complicated. Nevertheless, simply the $(d-1)$-volume of an appropriate hyperplane section of the body can be used for a quite precise (as confirmed by examples and sharpness) upper bound on Christoffel function.

Let ${\rm Vol}_{k}(\cdot)$ be the $k$-dimensional volume. Now we can state our main result for higher dimensions.
\begin{theorem}\label{mult}
Suppose a convex body $D\subset\R^d$ contains a ball of radius $r$ and is contained in a ball of radius $R$. For any $\x\in D\setminus \partial D$, let $\u\in \R^d$ be a unit vector such that for some $\nu\ge 1$ 
\begin{equation}\label{delta-bound}
\delta:=\max\{t:\x+t\u\in D\} \le \nu \dist(\x,\partial D)
\end{equation}
 and the hyperplane passing through $\x+\delta\u$ with normal vector $\u$ is supporting to $D$.
If $\delta\ge \sigma n^{-2}$, $\sigma>0$, then
\begin{equation}\label{mult_est}
\lambda_n(D,\x) \le c(d,r,R,\nu,\sigma) n^{-d} \min\{\sqrt{\delta}, \delta^{1-d/2} {\rm Vol}_{d-1}(\{\y\in D: (\x-\y)\perp \u\})\}.
\end{equation}
\end{theorem}

\begin{remark}
	The required choice of $\u$ is always possible even with $\nu=1$, namely, when $\delta=\dist(\x,\partial D)$ and $\u$ such that $\x+\delta\u\in\partial D$. Allowing $\nu>1$ gives more flexibility in the choice of the direction $\u$, for example of such application see the proof of Theorem~\ref{thm_example_half_ball}.
\end{remark}

\begin{remark}
Remark similar to Remark~\ref{constructive} holds about Theorem~\ref{mult}.
\end{remark}

\begin{remark} The main idea of the proofs of both of the main results Theorem~\ref{main_lower_thm} and Theorem~\ref{mult} relies on application of~\cite[Theorem~6.3]{Di-Pr}, which uses a parallelotop (an affine image of the cube) containing the body. Informally, in geometric language, one seeks such a circumscribed parallelotop having small volume and one of the vertices close to the point where the estimate of Christoffel function is sought. Our proofs describe an efficient way of constructing the corresponding affine transform of the cube, and so provide a relief from the need to optimize over a very large family of possible affine transforms for which~\cite[Theorem~6.3]{Di-Pr} is applicable.
\end{remark}

The common part of~\eqref{main_lower_est} and~\eqref{mult_est} is valid under somewhat milder hypothesis. Let us state this separately as a lemma.
\begin{lemma}\label{nd_root_delta}
	Suppose a convex body $D\subset\R^d$ is contained in a ball of radius $R$. For any $\x\in D\setminus \partial D$, let $\u\in \R^d$ be a unit vector such that 	
\begin{equation}\label{delta-in-lemma}
	\delta:=\max\{t:\x+t\u\in D\} \le \nu \dist(\x,\partial D)
\end{equation}
	for some $\nu\ge1$.
	If $\delta\ge \sigma n^{-2}$, $\sigma>0$, then
	\begin{equation}\label{root_delta}
	\lambda_n(D,\x)\le c(R,d,\nu,\sigma) n^{-d}\sqrt{\delta}.
	\end{equation}
\end{lemma}

\begin{remark}\label{rem-compare-requirements}
Note that under the conditions of Theorem~\ref{main_lower_thm}, due to convexity of $D$ (similarly to~\eqref{li_delta}), we have $\frac{\delta r}{2R}B^2+\x\subset D$, hence $\delta\le \frac{2R}{r} \dist(\x,\partial D)$, implying~\eqref{delta-in-lemma}. The requirements in Theorem~\ref{mult} are clearly stronger than those in Lemma~\ref{nd_root_delta}. 
\end{remark}

\begin{remark}
 The inequality~\eqref{root_delta} can be considered a domain independent upper bound on $\lambda_n$ using only $\delta$. It will be sharp near points where the boundary is sufficiently smooth, see, e.g., Proposition~\ref{prop-example} or Proposition~\ref{prop-example-balpha}. A domain independent lower bound is given in~\cite[Theorem~4]{Kr}. That lower bound is sharp in the opposite situation, say, near ``corners'' of the domain or near vertices of polytopes.
\end{remark}

\begin{remark}
	Without modifications to the proofs, one can replace the condition~\eqref{delta-bound} or~\eqref{delta-in-lemma} with somewhat less demanding
	\[ 
	\delta\le\nu\dist(\x,\partial D\cap\{\y\in D:(\y-\x)\perp \u\}).
	\]
\end{remark}

\begin{remark}
	For anisotropic convex domains $D$, it may be beneficial to apply Theorem~\ref{main_lower_thm} or Theorem~\ref{mult} not for $D$, but for $T(D)$, where $T$ is an affine transform satisfying $B^d\subset T(D)\subset dB^d$. The existence of such $T$ is guaranteed by John's ellipsoid theorem~\cite{Jo}, more precisely, $T$ can be chosen so that $T^{-1}(B^d)$ is the ellipsoid of maximum volume in $D$. It is straightforward to track how Christoffel function changes after the affine transform of the domain, see~\eqref{affine}. Note, however, that~\eqref{main_lower_est} and~\eqref{mult_est} are not invariant under affine transforms.
\end{remark}

Theorem~\ref{mult} is sharp in the class of convex bodies if one only uses measurements $\delta$ and $ {\rm Vol}_{d-1}( \{\y\in D:(\y-\x)\perp \u\}$ as the next result shows. Note that~\eqref{delta-bound} implies that
\[
\tfrac{\delta}{\nu} B^d +\x \subset D, \quad\text{so}\quad
c(d,\nu)\delta^{d-1} \le {\rm Vol}_{d-1}( \{\y\in D:(\y-\x)\perp \u\} ).
\]

\begin{theorem}\label{mult_sharp}
For any $d\ge2$ there exist positive constants  $\beta_1$ and $\beta_2$ which depend only on $d$ such that for any positive $\delta$ and $v$ with $\beta_1 \delta^{d-1} < v < \beta_2$, one can find a convex body $D\subset\R^d$ and $\x\in D$ satisfying $B^d\subset D\subset 3B^d$, $\delta=\dist(\x,\partial D)$, $v={\rm Vol}_{d-1}\{\y\in D:(\y-\x)\perp \u\})$ where $\u$ is a unit vector such that $\x+\delta\u\in\partial D$ (so~\eqref{delta-bound} is satisfied even with $\nu=1$ and the hyperplane through $\x+\delta\u$ with normal vector $\u$ is supporting to $D$), and for any $n$ with $\delta>\sigma n^{-2}$, $\sigma>0$, the following inequality holds:
\begin{equation*}
\lambda_n(D,\x)\ge c(d,\sigma) n^{-d}\min\{\sqrt{\delta}, \delta^{1-d/2} v\}.
\end{equation*}
\end{theorem}

\begin{remark}\label{bulk}
Our estimates of behavior of Christoffel function are valid and stated for the uniform weight only. However, the implied bounds on asymptotics of Christoffel function can be combined with the universality in the bulk results of~\cite{Kr-Lu} to obtain upper bounds on asymptotics of Christoffel function for positive continuous weights on the same domain.
\end{remark}

The titles of the following sections are self-explanatory. A reader not interested in the proofs is encouraged to proceed directly to Section~\ref{sect_examples} for examples of applications of main results. 


\section{Tools and auxiliary results}\label{sect_tools}

We begin with two important ingredients used frequently in the proofs here.
For two domains satisfying $D_1\subset D_2\subset \R^d$, by~\eqref{def_lambda} one observes that
\begin{equation}
\label{compare}
\lambda_n(D_1,\x)\le \lambda_n(D_2,\x), \quad \x\in D_2.
\end{equation}
For an affine transform $T\x=\x_0+A\x$ of $\R^d$, where $\x_0\in\R^d$ and $A$ is an $n\times n$ matrix, we will write $\det T=\det A$. Unless specified otherwise, any affine transform below is assumed to be non-degenerate, i.e., $\det T\ne 0$. From~\eqref{def_lambda} it is straightforward to compute that
\begin{equation}
\label{affine}
\lambda_n(TD,T\x)=\lambda_n(D,\x)|\det T|, \quad \x\in D.
\end{equation}
Although both~\eqref{compare} and~\eqref{affine} are directly applicable only for the uniform weight, they may lead to asymptotic results for other weights (see Remark~\ref{bulk}).

The crucial tool for upper bounds is~\cite[Theorem~6.3]{Di-Pr} which we now restate as a lemma.
\begin{lemma}\label{lower_lemma}
Let $D\subset\R^d$ be a compact set, $\y=(y_1,\dots,y_d)\in[-1,1]^d$, $T$ be an affine transformation of $\R^d$ such that $D\subset T([-1,1]^d)$ and $T\y\in D$. Then
\[
\lambda_n(D,T\y)\le c |\det T| \rho_n(y_1)\rho_n(y_2)\dots \rho_n(y_d)
\]
where $c>0$ depends only on $d$, and $\rho_n(y)=n^{-2}+n^{-1}\sqrt{1-y^2}$.
\end{lemma}

To establish sharpness of our main results, the lower bound in the following relation will be useful:
\begin{equation}
\label{ball}
\lambda_n(B^d,\x)\approx c(d,\sigma) n^{-d} \sqrt{1-|\x|^2}, \quad \x\in (1-\sigma n^{-2})B^d.
\end{equation}
Note that the asymptotics given in~\eqref{ball_as} does not imply~\eqref{ball} as that asymptotic relation is not known to be uniform. It is feasible that the methods of the proofs in~\cite{Bo-elal} or in~\cite{Xu} can be used to obtain~\eqref{ball}, however, such an approach would be very technical. Below we will provide rather elementary proof of the lower bound on $\lambda_n(B^d,\x)$ in~\eqref{ball}. This lower bound can also be derived from the positive cubature formula~\cite[Th.~6.3.3, p.~138]{Da-Xu} on the sphere and the connection to the ball~\cite[Ch.~11.1, p.~265]{Da-Xu}. The corresponding upper bound in~\eqref{ball} immediately follows from Lemma~\ref{lower_lemma} with $T$ chosen to be the identity and $\y$ located on one of the coordinate axes.

\begin{remark}
	If $\x$ is an interior point of a bounded $D\subset \R^d$ which is sufficiently far from the boundary, namely, if $rB^d+\x\subset D\subset RB^d+\x$, then~\eqref{compare}, \eqref{affine} and~\eqref{ball} provide $\lambda_n(D,\x)\approx c(d,r,R) n^{-d}$. Moreover, if $D\subset\R^d$ is a convex body containing a ball of radius $r$ and contained in a ball of radius $R$, then for $\delta=\dist(\x,\partial D)$ convexity implies $\frac{\delta r}{2R} B^d+\x\subset D$, so $\lambda_n(D,\x)\approx c(\delta,d,r,R)n^{-d}$.
\end{remark}

\begin{remark}
	Proof of~\cite[Theorem~1]{Kr} uses the asymptotics~\eqref{ball_as}. Instead, for the uniform weight, one can use~\eqref{ball} and obtain a version of~\cite[Theorem~1]{Kr} which bounds the behaviour of Christoffel function rather than its asymptotics.
\end{remark}

\begin{proof}[Proof of lower bound in~\eqref{ball}.]
Due to rotation invariance of $B^d$, let us assume $\x=(1-\delta,0,\dots,0)\in (1-\sigma n^{-2}) B^d$. If $\delta\ge1/2$, the required result follows from~\cite[Theorem~4.1, (4.4)]{Di-Pr} or from comparison with cube or simplex. Assuming $\delta<1/2$, consider any polynomial $P\in\P_{n,d}$ satisfying $P(\x)=1$ and $\|P\|^2_{L_2(B^d)} =  \lambda_n(B^d,\x)$. We need to show that $\|P\|_{L_2(B^d)}^2\ge c(d)\sqrt{\delta} n^{-d}$. Let $M:=\|P\|_{L_\infty((1-\delta/2)B^d)}\ge 1$ which is attained at a point $\y\in (1-\delta/2)B^d$. We obtain
\[
\lambda_n(B^d,\y) = \min_{Q\in \P_{n,d}, |Q(\y)|=1} \|Q\|_{L_2(B^d)}^2 \le \frac1{M^2}\|P\|^2_{L_2(B^d)} = \frac1{M^2} \lambda_n(B^d,\x),
\]
so by~\cite[Theorem~4.1, (4.3)]{Di-Pr} (or by forthcoming Lemma~\ref{almost-increase} yielding a somewhat larger constant), \eqref{compare} and~\eqref{affine}, we conclude that
\begin{align*}
M^2 & \le \frac{\lambda_n(B^d,\x)}{\lambda_n(B^d,\y)} \le
\frac{\lambda_n(B^d,(1-\delta,0,\dots,0))}{\lambda_n(B^d,(1-\delta/2,0,\dots,0))} \\
& \le \frac{\lambda_n(B^d,(1-\delta,0,\dots,0))}{\lambda_n(1/2B^d+(1/2,0,\dots,0),(1-\delta/2,0,\dots,0))} = 2^d.
\end{align*}
Let
\[
S:=\{(x_1,\dots,x_d):|(x_1,\dots,x_d)|=1,\ x_1\ge \cos(1/(2nM))\}
\]
be the spherical cap on the unit sphere centered at $(1,0,\dots,0)$ of angle $1/(2nM)$. We claim that
\begin{equation}\label{cap1}
P((1-\delta)\z)\ge \frac12, \quad \z\in S.
\end{equation}
Indeed, if $(1-\delta)\z\ne \x$, let $\omega$ be the two dimensional circle which is the intersection of the sphere of radius $1-\delta$ centered at the origin of $\R^d$ and the two dimensional plane through the origin and the points $\x$ and $(1-\delta)\z$. We can consider the restriction of $P$ to $\omega$ as a trigonometric polynomial of degree at most $n$. The derivative of $P$ along $\omega$ is at most $nM$ by Bernstein's inequality. Since the angle between the vectors $\x$ and $(1-\delta)\z$ is at most $1/(2nM)$ and $P(\x)=1$, the required inequality~\eqref{cap1} follows.

For arbitrary $\z\in S$ we now consider the segment $\{(1-\delta/2)t\z:t\in[-1,1]\}\subset (1-\delta/2)$. The univariate polynomial $p(t):=P((1-\delta/2)t\z)$ satisfies
\[
p\left(\frac{1-\delta}{1-\delta/2}\right)\ge \frac12 \quad\text{and}\quad \|p\|_{L_\infty[-1,1]}\le M.
\]
By Bernstein's inequality,
\[
|p'(t)|\le \frac{n}{\sqrt{1-t^2}}M \le \frac{2nM}{\sqrt{\delta}}\quad \text{if} \quad |t|\le \frac{1-\delta/2}{1-\delta},
\]
so for the interval $I$ of length $\frac{\sqrt{\delta}}{8nM}$ having the right endpoint $\frac{1-\delta/2}{1-\delta}$, we have $p(t)\ge\frac14$ when $t\in I$. So,
\[
P((1-\delta/2)t\z)\ge \frac14, \quad t\in I, \quad \z\in S,
\]
and recalling that $M^2\le 2^d$, it is not hard to see that the measure of $\{(1-\delta/2)t\z: t\in I, \z\in S\}$ is at least $c(d)\sqrt{\delta}n^{-d}$, which implies the required lower bound on $\|P\|_{L_2(B^d)}^2$.
\end{proof}

Now we will prove Proposition~\ref{boundary} and a lemma that can be of independent interest, as it shows that the Christoffel function is nearly decreasing on rays towards the boundary of the domain.
\begin{lemma}\label{almost-increase}
	Let $D$ be a star-like body in $\R^d$. If $\x\in D$ and $0<\mu<1$, then
	\[
	\lambda_n(D,\x)\le \mu^{-d} \lambda_n(D,\mu\x).
	\]
\end{lemma}
\begin{proof}
	By~\eqref{compare} and~\eqref{affine}, $\ds\lambda_n(D,\x)\le \lambda_n(\mu^{-1}D,\x)=\mu^{-d}\lambda_n(D,\mu \x)$.
\end{proof}

\begin{proof}[Proof of Proposition~\ref{boundary}.]
	The required upper bound of $\lambda_n(D,\x)$ is immediate by Lemma~\ref{almost-increase} (an absolute constant in the bound can be verified by computations), so it remains to prove the lower bound. Let $P\in\P_{n,d}$ be such that $P(\x)=1$ and $\lambda_n(D,\x)=\|P\|^2_{L_2(D)}$. Take $I:=\{t\x:t\in[1/2,1]\}$, and let $M:=\|P\|_{L_\infty(I)}\ge 1$ be attained at a point $\y\in I$. Since
	\[
	\lambda_n(D,\y) = \min_{Q\in \P_{n,d}, |Q(\y)|=1} \|Q\|_{L_2(D)}^2 \le \frac1{M^2}\|P\|^2_{L_2(D)} = \frac1{M^2} \lambda_n(D,\x),
	\]
	by Lemma~\ref{almost-increase} and choice of $I$, we obtain $M^2\le 2^d$. Applying Markov's inequality to $P$ on $I$, we get that $P(\mu\x)\ge 1/2$ for any $\mu\in[1-(8Mn^2)^{-1},1]\supset [1-c(d)n^{-2},1]$, so
	\[
	\lambda_n(D,\mu \x) = \min_{Q\in \P_{n,d}, |Q(\mu\x)|=1} \|Q\|_{L_2(D)}^2 \le \frac1{(P(\mu\x))^2}\|P\|^2_{L_2(D)} \le 4 \lambda_n(D,\x),
	\]
	and the required lower bound on $\lambda_n(D,\x)$ is established.
\end{proof}

\section{Applications and examples}\label{sect_examples}

First we illustrate what is the behavior of Christoffel function when the boundary is sufficiently smooth.
\begin{proposition}\label{prop-example}
	Suppose $D\subset \R^d$ is a convex body for which $\partial D$ is a $(d-1)$-dimensional $C^2$ submanifold in $\R^d$ (in the sense of differential geometry). For any interior point $\x\in D$, let $\delta:=\dist(\x,\partial D)$.  If $\sigma n^{-2}<\delta<1$, $\sigma>0$, then
	\[
	\lambda_n(D,\x)\approx c(D,\sigma)n^{-d}\sqrt{\delta}.
	\]
\end{proposition}
\begin{proof}
	The upper estimate of $\lambda_n(D,\x)$ is guaranteed by Lemma~\ref{nd_root_delta} (even without $C^2$ assumption). The lower estimate follows from~\eqref{compare} and~\eqref{ball} by considering an inscribed ball of radius $c(D)$ tangent to $\partial D$ at the point $(\x+\delta B^d) \cap \partial D$, which exists due to $C^2$ smoothness. 
\end{proof}
Crucial for the lower bound in Proposition~\ref{prop-example} is the property that a ball of a fixed radius ``rolls freely'' inside $D$. One can refer to~\cite{Wa} for further discussion of this property and alternative equivalent conditions on $\partial D$.

The second example was, in fact, the main motivation for this work. It is concerned with estimating the behaviour of Christoffel function for the unit balls in $l_\alpha$ metric, $1<\alpha<2$, which serve as examples of bodies that are ``between'' the smooth $C^2$ case of a ball ($\alpha=2$) and the non-smooth case of a polytope ($\alpha=1$). Namely, we denote (for Euclidean balls we have $B^d=B^d_2$)
\[
B^d_\alpha:=\{(x_1,\dots,x_d):|x_1|^\alpha+\dots+|x_d|^\alpha\le 1\}, \quad 1\le\alpha<\infty.
\]
As a simple application of Theorem~\ref{mult}, we show that the extra logarithmic factor in~\cite[Theorem~2]{Kr} can be removed.
\begin{corollary}
Suppose $1\le\alpha\le 2$, and let $\gamma(\alpha,d)=\frac12+\frac{(d-1)(2-\alpha)}{2\alpha}$. If $\sigma n^{-2}<\delta<1$, $\sigma>0$, then
\[
\lambda_n(B^d_\alpha,(1-\delta,0,\dots,0))\le c(d,\alpha,\sigma)n^{-d}\delta^{\gamma(\alpha,d)}.
\]
\end{corollary}
\begin{proof}
We will apply Theorem~\ref{mult}. Since $\frac{1}{\sqrt{d}}B^d\subset B^d_1\subset B^d_\alpha\subset B^d$, we can take $r=\frac{1}{\sqrt{d}}$ and $R=1$.  With $\tilde\delta:=(1-(1-\delta)^\alpha)^{1/\alpha}\approx c(\alpha) \delta^{1/\alpha}$, for $\x=(1-\delta,0,\dots,0)$ we have $\delta=\dist(\x,\partial B^d_\alpha)$ and $\u=(1,0,\dots,0)$ clearly satisfies~\eqref{delta-bound} and the other required condition. The hyperplane section $\{\y\in D:(\y-\x)\perp \x \}$ is exactly $\tilde\delta B^{d-1}_\alpha$, which has $(d-1)$-volume $c(d,\alpha)\delta^{(d-1)/\alpha}$. Therefore, by~\eqref{mult_est},
\[
\lambda_n(D,\x)\le c(d,\alpha,\sigma) n^{-d} \delta^{1-d/2} \delta^{(d-1)/\alpha}=c(d,\alpha,\sigma) \delta^{\gamma(d,\alpha)}.
\]
\end{proof}
Matching asymptotic lower bound for the so called $C^\alpha$ domains (which include $B^d_\alpha$) was established in~\cite[Theorem~1]{Kr}. The lower bound of~\cite[Theorem~1]{Kr} provides a ``worst-case'' exponent $\gamma(d,\alpha)$ for $C^\alpha$ domains, which is sharp on the segment joining the origin and the ``most singular'' points such as $(1,0,\dots,0)$ for $B^d_\alpha$. The actual behavior of Christoffel function in other locations of the domain can be very different. In particular, near the points where the boundary is smooth, we can get the exponent corresponding to $\alpha=2$. Let us illustrate this pointwise phenomenon for the diagonal direction of $B^2_\alpha$, $1\le \alpha\le 2$.
\begin{proposition}\label{prop-example-balpha}
For any $1\le \alpha\le 2$, let $\x=(2^{-1/\alpha},2^{-1/\alpha})$ (which belongs to $\partial B^2_\alpha$). If  $\sigma n^{-2}<\delta<1/2$, $\sigma>0$, then
\[
\lambda_n(B^2_\alpha,(1-\delta)\x)\approx c(\alpha,\sigma) n^{-2} \sqrt{\delta}.
\]
\end{proposition}
\begin{proof}
The upper estimate of $\lambda_n(B^2_\alpha,(1-\delta)\x)$ is immediate by Theorem~\ref{main_lower_thm} (or simply by Lemma~\ref{nd_root_delta}). For the lower estimate, inscribe a disc of radius $c(\alpha)$ into $B^2_\alpha$ tangent to the boundary at $\x$ with the center on the line $x_1=x_2$, and use~\eqref{compare} and~\eqref{ball}. We hope the reader will forgive us the omission of technical details here.
\end{proof}
It would be interesting to compute the actual pointwise behavior of $\lambda_n(B^2_\alpha,\x)$ for arbitrary $\x$. 
\begin{conjecture}
For any $1<\alpha<2$, any $\x\in B^2_\alpha$, let $\delta:=\dist(\x,\partial B^2_\alpha)$ and $l_i:=\max\{t:\x+(-1)^it\v\in B^2_\alpha\}$, $i=1,2$, where $\u$ is such that $\x+\delta\u\in\partial D$ and $\v$ is one of the two unit vectors orthogonal to $\u$. If $\delta>\sigma n^{-2}$, $\sigma>0$, then
\begin{equation}\label{eqn-conj}
\lambda_n(B^2_\alpha,\x) \approx c(\alpha,\sigma) n^{-2} \sqrt{l_1l_2}.
\end{equation}
\end{conjecture}
The upper bound of $\lambda_n(B^2_\alpha,\x)$ in~\eqref{eqn-conj} is valid by Theorem~\ref{main_lower_thm} (one can show that $\min\{l_1l_2,\delta\}\approx c(\alpha)l_1l_2$ using circumscribed discs), so it only remains to establish the lower bound. We also believe that  $l_1\approx c(\alpha)l_2$ in these settings.

There are some domains which do not properly fall into the proposed $C^\alpha$ classification of~\cite{Kr}. One such example is half-ball in $\R^3$
\[
B^3_+:=\{(x_1,x_2,x_3)\in B^3:x_3\ge0\}
\]
for which the behaviour of the Christoffel function on the ``rim'' $\{(x_1,x_2,0):x_1^2+x_2^2=1\}$ was found in~\cite[Section~9]{Di-Pr}. Below we make a ``diagonal step'' inside the domain from the rim and compute the order of the Christoffel function. 
\begin{theorem}\label{thm_example_half_ball}
If $\sigma n^{-2}<\mu<1/3$, $\sigma>0$, then
\[
\lambda_n(B^3_+,(1-\mu,0,\mu/4))\approx c(\sigma) n^{-3}\mu.
\]
\end{theorem}
\begin{proof}
We begin with the upper bound for which Theorem~\ref{mult} will be used. Clearly, we can take $r=1/2$ and $R=1$. With $\x=(1-\mu,0,\mu/4)$, we can choose $\u=\frac4{\sqrt{17}}(1,0,-\frac14)$, then $\delta=\frac{\sqrt{17}}{4} \mu$ and the two dimensional plane through $\x+\delta\u=(1,0,0)$ with normal vector $\u$ is supporting to $B^3_+$. It is easy to see that
\[
\dist(\x,\partial B^3_+)=\min\{\mu/4,1-|\x| \}=\mu/4,
\]
so~\eqref{delta-bound} holds. All the conditions of Theorem~\ref{mult} are fulfilled. With $V(D,\x):=\{\y\in D:(\y-\x)\perp \x \}$, the required bound follows from~\eqref{mult_est} if we show that ${\rm Vol}_{2}(V(D,\x))\le c \mu^{3/2}$. If $(x,y)$ are Cartesian coordinates in the two dimensional plane containing $V(D,\x)$ chosen so that $\x$ has the coordinates $(x,y)=(0,0)$, $x$-axis is parallel to $x_2$ axis and $y$-axis points in the positive direction of $x_3$ axis (upwards), then it is not hard to see that in this new coordinate system
\[
V(D,\x)\subset \left[-\sqrt{\tfrac{17}{8}\mu},\sqrt{\tfrac{17}{8}\mu}\right]\times [-\tfrac{\sqrt{17}}{16}\mu,\sqrt{17}\mu],
\]
which immediately implies ${\rm Vol}_{2}(V(D,\x))\le c \mu^{3/2}$, and we are done with the upper estimate.

For the lower bound, we note that it was established in the proof of~\cite[Lemma~9.6]{Di-Pr} that the affine transform
\[
T(x_1,x_2,x_3)=\left(1-\frac{1+\mu-x_1}2,\frac{x_2}8,\frac{\sqrt{3\mu}}{10}x_3+\frac{1+\mu-x_1}8\right)
\]
satisfies $TB^3\subset B^3_+$ and $|\det T|\ge\frac{\sqrt{3\mu}}{160}$ (when $0<\mu<1/3$). Therefore, by~\eqref{compare}, \eqref{affine} and~\eqref{ball},
\begin{align*}
\lambda_n(B^3_+,(1-\mu,0,\mu/4)) &= \lambda_n(B^3_+,T(1-\mu,0,0))\ge \lambda_n(T(B^3),T(1-\mu,0,0)) \\
&= |\det T| \lambda_n(B^3,(1-\mu,0,0)) \ge c(\sigma) \sqrt{\mu}\cdot n^{-3}\sqrt{\mu} = c(\sigma)n^{-3}{\mu}.
\end{align*}
\end{proof}

\begin{remark}
	Theorem~\ref{thm_example_half_ball} can be generalized to higher dimensions using the same technique.
\end{remark}


\section{Proofs of the main results}\label{sect_proofs}

\begin{proof}[Proof of Lemma~\ref{nd_root_delta}.]
	We can assume $D\subset RB^d$, $\x=(a,0,\dots,0)$, $\u=(1,0,\dots,0)$. Let $\w=(w_1,\dots,w_d)$, $w_1\ge 0$, be the unit normal vector of a supporting hyperplane $\kappa$ to $D$ at $(a+\delta,0,\dots,0)$. Using $(t_1,\dots,t_d)$ as coordinates, the equation of $\kappa$ is
	\[
	(t_1-a-\delta)w_1+t_2w_2+\dots+t_dw_d=0.
	\]
	The condition $\dist(\x,\partial D)\ge \frac{\delta}{\nu}$ means that $\frac{\delta}{\nu} B^d+\x\subset D$, so from the above equation, for any $(t_2,\dots,t_d)\in \frac{\delta}{\nu} B^{d-1}$, we must have $\delta w_1\ge t_2w_2+\dots+t_d w_d$. Choose $t_j=\frac{\delta}{\nu}w_j$ for $j\ge2$, then $\delta w_1\ge \frac{\delta}{\nu}(1-w_1^2)\ge \frac{\delta}{\nu}(1-w_1)$, so $w_1\ge\frac1{\nu+1}$. The affine transform $T(z_1,\dots,z_d)=(t_1,\dots,t_d)$ defined by
	\begin{align*}
	t_1 &= \frac{(a+\delta+\frac{Rd}{w_1})z_1+a+\delta-\frac{Rd}{w_1}}2-\frac{R}{w_1}(w_2z_2+\dots+w_dz_d), \\
	t_j &= Rz_j, \quad j\ge2,
	\end{align*}
	maps the hyperplanes $z_j=\pm 1$ to the hyperplanes $t_j=\pm R$, $j\ge2$, the hyperplane $z_1=1$ to $\kappa$ and the hyperplane $z_1=-1$ to the hyperplane
	\[
	\left(t_1+\frac{Rd}{w_1}\right)w_1+t_2w_2+\dots+t_dw_d=0,
	\]
	which does not intersect $[-R,R]^d$ (as $|w_j|\le1$). Therefore,
	\[
	D \subset (RB^d)\cap\{(t_1,\dots,t_d):(t_1-a-\delta)w_1+t_2w_2+\dots+t_dw_d\le0\} \subset T([-1,1]^d).
	\]
	As $|\det T|=R^{d-1}(a+\delta+\frac{Rd}{w_1})/2<c(R,d,\nu)$, with $(y_1,\dots,y_d)=T^{-1}(a,0,\dots,0)$ one computes that $1-y_1=\frac{2w_1\delta}{Rd}<\frac{2}{Rd}\delta$ and $y_j=0$, $j\ge2$, so by Lemma~\ref{lower_lemma},
	\[
	\lambda_n(D,\x)\le c |\det T| \rho_n(y_1)  \dots \rho_n(y_d) \le c(R,d,\nu) n^{1-d} \rho_n(1-\tfrac{2}{Rd}\delta) \le c(R,d,\nu,\sigma) n^{-d} \sqrt{\delta},
	\]
	where in the last step $\delta>\sigma n^{-2}$ was used.
\end{proof}

\begin{proof}[Proof of Theorem~\ref{main_lower_thm}.]
In view of Lemma~\ref{nd_root_delta} and Remark~\ref{rem-compare-requirements}, we only need to establish that
\begin{equation}\label{root_l1l2}
\lambda_n(D,\x)\le c(r,R,\sigma) n^{-2} \sqrt{l_1l_2}.
\end{equation}
Using~\eqref{affine} and applying a translation and a rotation, if necessary, we can assume that $\x=(0,a)$, $\u=(0,1)$, $\v=(1,0)$, $rB^2\subset D\subset 2RB^2$, so $(0,a+\delta)\in\partial D$, $((-1)^il_i,a)\in\partial D$, $i=1,2$. Further, we can assume that $l_i<r$ for both $i=1,2$, as otherwise by~\eqref{li_delta} we have $l_1l_2\ge c(r,R) \delta$ and~\eqref{root_l1l2} follows from~\eqref{root_delta}. Let $q_i$, $i=1,2$, be the line passing through $((-1)^il_i,a)$ and $(0,a+\delta)$. Denote by $r_i$, $i=1,2$, the line through $((-1)^i r,0)$ and $((-1)^il_i,a)$. For a line $q$ not containing the origin, we denote by $q^+$ the half-plane bounded by $q$ and containing the origin. Since $(\pm r,0)\in D$ and $D$ is convex, we have
\[
D\subset q_1^+ \cup (r_1^+\cap q_2^+) \quad\text{and}\quad D\subset q_2^+ \cup (r_2^+\cap q_1^+).
\]
Therefore, if $P_i$ is the point of intersection of $r_i$ and $q_{3-i}$ (note that due to $l_i<r$ this point $P_i$ is located in the $(3-i)$-th quadrant), while $s_i$ is the line through $P_i$ parallel to $q_i$, $i=1,2$, we obtain $D\subset s_1^+\cap s_2^+$. Now we begin construction of an appropriate affine map $T:(z_1,z_2)\mapsto(x,y)$ with intention to apply Lemma~\ref{lower_lemma}. We require that $T$ is such that the lines $s_1$ and $s_2$ are the images of the lines $z_1=1$ and $z_2=1$ under $T$, respectively. Next we choose $t_i$, $i=1,2$, as the line parallel to $s_i$ such that $t_i$ is supporting to $D$ and the origin is between $t_i$ and $s_i$. Now $T$ will be uniquely defined, if, in addition to the above condition regarding $s_i$, we demand that $t_i$ is the image of the line $z_i=-1$ under $T$, $i=1,2$.

Let $\phi$ be the angle between the lines $q_1$ and $q_2$ which does not contain the $y$-axis. Using~\eqref{li_delta}, we have
\[
\phi=\arctan \frac\delta{l_1}+\arctan\frac{\delta}{l_2}\le 2\arctan c(r,R) \le \pi-c(r,R)
\]
and as $\arctan x > \frac x{2H}$ for $0<x<H$,
\[
\phi=\arctan \frac\delta{l_1}+\arctan\frac{\delta}{l_2}>c(r,R)\left(\frac\delta{l_1}+\frac\delta{l_2}\right),
\]
therefore,
\begin{equation}\label{angle}
\sin \phi \ge c(r,R)\left(\frac{\delta}{l_1}+\frac{\delta}{l_2}\right).
\end{equation}

As $l_2<r$, the line $r_2$ has negative slope, and the $x$-coordinate of $P_2$ is less than $l_2$. Hence, as the slope of $q_1$ is $\frac\delta{l_1}$, we use~\eqref{li_delta} to estimate that
\[
|P_2-(0,a+\delta)|<l_2\frac{\sqrt{\delta^2+l_1^2}}{l_1}<c(r,R) l_2.
\]
So the distance from $(0,a+\delta)$ to the line $s_2$ is less than $c(r,R) l_2 \sin\phi$. The distance from $(0,a)$ to $s_2$ is less than
$\delta+c(r,R) l_2\sin\phi <c(r,R) l_2\sin\phi$, where~\eqref{angle} was used. By similar arguments, the distance from $(0,a)$ to $s_1$ is at most $c(r,R) l_1 \sin\phi$.

Now it is clear that the distance between the parallel lines $s_i$ and $t_i$ is at least $2r$ (as $rB^2\subset D\subset T([-1,1]^2)$) and at most $c(R)$ (as $(0,a)\in D$ and $l_i<2R$), $i=1,2$. Hence,
\[
|\det T|=\frac{\rm{Vol}_2(T([-1,1]^2))}{\rm{Vol}_2([-1,1]^2)}\le\frac{c(R)}{\sin \phi}.
\]
Defining $(y_1,y_2)=T^{-1}(0,a)$, the bounds for the distances from $(0,a)$ to $s_1$ and to $s_2$ imply
\[
1-y_i\le c(r,R)l_i \sin\phi, \quad i=1,2.
\]
Since $\delta>\sigma n^{-2}$,
\[
\rho_n(y_i)\le \rho_n\left(1-c(r,R)l_i \sin\phi\right) \le c(r,R,\sigma) n^{-1} \sqrt{l_i \sin\phi}.
\]
We are ready to apply Lemma~\ref{lower_lemma} and obtain~\eqref{root_l1l2}. Indeed, using the above inequalities,
\[
\lambda_n(D,\x)\le c |\det T| \rho_n(y_1) \rho_n(y_2) \le c(r,R,\sigma) \frac{1}{\sin\phi} n^{-2} \sqrt{l_1 \sin\phi} \sqrt{l_2 \sin\phi} = c(r,R,\sigma) n^{-2}\sqrt{l_1 l_2},
\]
which completes the proof of the theorem.
\end{proof}

\begin{proof}[Proof of Theorem~\ref{mult}.]
Due to Lemma~\ref{nd_root_delta} and Remark~\ref{rem-compare-requirements}, we only need to show that
\[
\lambda_n(D,\x) \le c(d,r,R,\nu,\sigma) n^{-d} \delta^{1-d/2} {\rm Vol}_{d-1}(V(D,\x)),
\]
where $V(D,\x):=\{\y\in D: (\x-\y)\perp \u\}$.
We assume that $D\subset RB^d$, $\x=(a,0,\dots,0)$, $\u=(1,0,\dots,0)$ and $\y:=(a+\delta,0,\dots,0)\in \partial D$ while the hyperplane $\{(x_1,\dots,x_d):x_1=a+\delta\}$ is supporting to $D$ at $\y$.  In these settings, $V(D,\x)=\{(x_1,\dots,x_d)\in D:x_1=a\}$.

Before proceeding, we need some preliminaries. A $(d-1)$-simplex is the closed convex hull of $d$ points (vertices) not all lying in a $(d-2)$-dimensional plane. It has $d$ facets ($(d-2)$-dimensional faces), each facet is a $(d-2)$-simplex. The centroid of a $(d-1)$-simplex $S$ with vertices $\z_1,\dots,\z_d$ is $\z:=\frac1d(\z_1+\dots+\z_d)$. Since $\z_j=\frac1{d-1}\sum_{i\ne j}(-(d-1)(\z_i-\z)+\z)$, the image of $S$ under the homothety with coefficient $-(d-1)$ with respect to $\z$ is a simplex $S'=-(d-1)(S-\z)+\z$ containing $S$ such that every vertex $\z_j$ of $S$ belongs to (and is the centroid of) the facet of $S'$ which is parallel to the facet of $S$ not containing $\z_j$. Another property we need is that for any point $\z'\in S$ the homothety of $S$ with coefficient $2$ with respect to $\z'$ is contained in the homothety of $S$ with coefficient $d+1$ with respect to $\z$, i.e., that
\begin{equation}\label{sim}
2(S-\z')+\z'\subset (d+1)(S-\z)+\z.
\end{equation}
Indeed, for any $\w\in S$ we can write $\w=\sum_i \alpha_i \z_i$ while $\z'=\sum_i \beta_i \z_i$, where $\alpha_i$, $\beta_i$, $i=1,\dots,d$, are non-negative and $\sum_i\alpha_i=\sum_i\beta_i=1$. Then
\[
\frac{2\w-\z'+d\z}{d+1}=\sum_{i=1}^d\frac{2\alpha_i-\beta_i+1}{d+1}\z_i
\]
is a convex combination of $\z_i$, thus, it belongs to $S$ and~\eqref{sim} is proved.

Let $S$ be a $(d-1)$-simplex of largest $(d-1)$ volume contained in $V(D,\x)$. Let $\z$ be the centroid of $S$. Due to maximality of the $(d-1)$ volume, the homothety of $S$ with coefficient $-(d-1)$ with respect to $\z$, i.e., the simplex $S_1=-(d-1)(S-\z)+\z$ contains $V(D,\x)$. Indeed, assuming to the contrary that a point $\z'$ of $V(D,\x)$ is outside of $S_1$, we can find a $(d-2)$-dimensional plane $\tau$ containing a facet of $S_1$ such that $S$ and $\z'$ are separated by $\tau$. By what was established above, $\tau$ contains one vertex of $S$ and the remaining $d-1$ vertices of $S$ belong to a $(d-2)$-dimensional plane $\tau'$ parallel to $\tau$. Due to separation, the distance from the point $S\cap\tau$ to $\tau'$ is less than the distance from $\z'$ to $\tau'$.  Therefore, the simplex with vertices $z'$ and $d-1$ vertices of $S$ belonging to $\tau'$ will have a larger volume than $S$ and will be contained in $V(D,\x)$, contradiction.

The inclusions $S\subset V(D,\x)\subset S_1$ imply ${\rm Vol}_{d-1}(V(D,\x))\approx c(d) {\rm Vol}_{d-1}(S)$. Consider the ``corner''
$
K_1=\{(1-t)\y+t\v: t\ge0,\ \v\in S_1\}
$
with ``vertex'' $\y$ consisting of all rays originating at $\y$ and passing through points of $S_1$. Alternatively, $K_1$ is the intersection of $d$ half-spaces containing the origin and determined by the hyperplanes which pass through $\y$ and a facet of $S_1$. We claim that
\begin{equation}\label{dom1}
D\cap\{(x_1,\dots,x_d):x_1\le a\} \subset K_1.
\end{equation}
Indeed, if $\w=(w_1,\dots,w_d)\in D$, $w_1\le a$, then $\y+t(\w-\y)\in V(D,\x)\subset S_1$, where $t=\frac{\delta}{a+\delta-w_1}$, so $\w\in K_1$.

Let $P(x_1,\dots,x_d):=(x_2,\dots,x_d)$ be the orthogonal projection along the first coordinate. Set $S_2:=(d+1)(S_1-\z)+\z$. Now we claim that
\begin{equation}\label{dom2}
P(\{(x_1,\dots,x_d)\in D: a\le x_1\le a+\delta\}) \subset P(S_2).
\end{equation}
We argue similarly to the proof of the previous inclusion, but now use $(a-\delta,0,\dots,0)\in D$ and convexity. So if $\w=(w_1,\dots,w_d)\in D$, $a\le w_1\le a+\delta$, then
\[
(1-t)(a-\delta,0,\dots,0)+t\w \in V(D,\x)\subset S_1,
\]
with $t=\frac{\delta}{\delta+w_1-a}\in[\frac12,1]$. Therefore, $P(\w)\in P(2S_1)\subset P(S_2)$, where we used $P(\x)=\zero\in P(S_1)$ and~\eqref{sim}.

Next we define the corner
\[
K_2=\{(1-t)(a+2\delta,0,\dots,0)+t(\v+(\delta,0,\dots,0)): t\ge0,\ \v\in S_2\}.
\]
We claim that $K_1\subset K_2$. Indeed, consider any point $(1-t)(a+\delta,0,\dots,0)+t\v$ of $K_1$, where $t\ge 0$ and $\v\in S_1$. Then $\frac{t}{t+1}P(\v)\in P(S_1)\subset P(S_2)$ due to $\zero\in P(S_1)$, so
\[
(1-t)(a+\delta,0,\dots,0)+t\v=(1-(t+1))(a+2\delta,0,\dots,0)+(1+t)(\tfrac{t}{t+1}\v+(\delta,0,\dots,0))\in K_2.
\]
We also have $[a,a+\delta]\times P(S_2)\subset K_2$, which is clearly seen from the definition of $K_2$ with $t\in [1,2]$. In summary, due to~\eqref{dom1}, \eqref{dom2} and the fact that $D\subset\{(x_1,\dots,x_d):x_1\le a+\delta\}$, we conclude that $D\subset K_2$.

Now we define an affine mapping $T$ of $\R^d$ so that the hyperplanes $z_i=1$, $i=1,\dots,d$, are mapped into the hyperplanes defining $K_2$, while the hyperplanes $z_i=-1$ are mapped into the hyperplanes supporting to $D$ so that $D\subset T([-1,1]^d)$ (this is possible since we established that $D\subset K_2$ while $D$ is bounded). Note that the distance from $\x=(a,0,\dots,0)$ to any of the hyperplanes $T(\{z_i=1\})$ is at most the distance from $\x$ to $(a+2\delta,0,\dots,0)=T(1,\dots,1)\in T(\{z_i=1\})$, which is $2\delta$. Recalling that $D$ contains a ball of radius $r$, we see that the distance between the hyperplanes $T(\{z_i=1\})$ and $T(\{z_i=-1\})$ is at least $2r$. Therefore, if $(y_1,\dots,y_d)=T^{-1}(\x)$, we have $1-y_i\le \frac{\delta}{r}$, $i=1,\dots,d$. Also, since $D$ is contained in a ball of radius $R$, we obtain that the distance between the hyperplanes $T(\{z_i=1\})$ and $T(\{z_i=-1\})$ is at most $2R+2\delta\le c(R)$.

Now our goal is to estimate $|\det T|$. Denote $S_3:=T([-1,1]^d)\cap\{x_1=a\}=K_2\cap\{x_1=a\}$ and $A:={\rm Vol}_{d-1}(S_3)=2^{d-1} {\rm Vol}_{d-1}(S_2) \approx c(d) {\rm Vol}_{d-1}(V(D,\x))$. We also need $V_i:= {\rm Vol}_{d-1}(T([-1,1]^d\cap\{z_i=1\}))$ the $(d-1)$-volumes of the facets of $T([-1,1]^d)$ and $V:= {\rm Vol}_{d}(T[-1,1]^d)$. Now the fact that the distance between the hyperplanes $T(\{z_i=1\})$ and $T(\{z_i=-1\})$ is at most $c(R)$ can be written as $\frac{V}{V_i}\le c(R)$. Let $\wt\x:=(a+2\delta,0,\dots,0)=T(1,\dots,1)$, and let $\x_i$ be the vertex of $S_3$ which does not belong to $T(\{z_i=1\})$, while $\wt \x_i$ be the vertex of $T([-1,1]^d)$ on the line joining $\wt\x$ and $\x_i$ different from $\wt\x$, $i=1,\dots,d$. Set $\alpha_i:=\frac{|\wt\x-\wt\x_i|}{|\wt\x-\x_i|}$. Using $\conv(\cdot)$ to denote the closed convex hull, we have
\begin{align*}
V &=d! \, {\rm Vol}_{d}(\conv(\{\wt \x,\wt\x_1,\dots,\wt\x_d\}))= d!\, {\rm Vol}_{d}(\conv(\{\wt \x,\x_1,\dots,\x_d\}))  \prod_{i=1}^d \alpha_i \\
&=2 (d-1)!\, \delta A \prod_{i=1}^d \alpha_i,
\end{align*}
and similarly
\begin{align*}
V_j &= (d-1)!\, {\rm Vol}_{d-1}(\conv(\{\wt \x,\wt\x_1,\dots,\wt\x_d\}\setminus\{\wt\x_j\})) \\
&= (d-1)!\, {\rm Vol}_{d-1}(\conv(\{\wt \x,\x_1,\dots,\x_d\}\setminus\{\x_j\})) \prod_{i\ne j} \alpha_i,
\end{align*}
so
\begin{equation}\label{vvj}
c(R)\ge \frac{V}{V_j}=\frac{2\delta A\alpha_j}{{\rm Vol}_{d-1}(\conv(\{\wt \x,\x_1,\dots,\x_d\}\setminus\{\x_j\}))}.
\end{equation}
Note that the orthogonal projection of $\conv(\{\wt \x,\x_1,\dots,\x_d\}\setminus\{\x_j\})$ onto $\{x_1=a\}$ is $\conv(\{\x,\x_1,\dots,\x_d\}\setminus\{\x_j\})$. Recall that $\frac{\delta}{\nu}B^d+\x\subset D$, hence, $\frac{\delta}{\nu}B^{d-1}\subset S_3$, so the distance from $\x$ to the $(d-2)$ dimensional plane containing $\conv(\{\x_1,\dots,\x_d\}\setminus\{\x_j\})$ is at least $\frac\delta\nu$. The distance between $\x$ and $\wt \x$ is $2\delta$, so we have that the angle between the hyperplane containing $\{\wt \x,\x_1,\dots,\x_d\}\setminus\{\x_j\}$ and the hyperplane containing $\{\x,\x_1,\dots,\x_d\}\setminus\{\x_j\}$ is at most $\arctan\frac\nu2$, so the cosine of this angle is at least $\cos(\arctan\frac\nu2)=:\tilde c>0$. Thus, continuing~\eqref{vvj}, we have
\begin{align*}
c(R)\delta A \alpha_j &\le {\rm Vol}_{d-1}(\conv(\{\wt \x,\x_1,\dots,\x_d\}\setminus\{\x_j\})) \\ & \le \frac{{\rm Vol}_{d-1}(\conv(\{\x,\x_1,\dots,\x_d\}\setminus\{\x_j\}))}{\tilde c}
\le \frac{A}{\tilde c}.
\end{align*}
Finally,
\[
|\det T|=\frac{V}{2^d}=c(d)\delta A \prod_{i=1}^d \alpha_i \le c(R,d,\nu)\delta A \prod_{i=1}^d \frac1{\delta} \le c(R,d,\nu) \delta^{1-d} {\rm Vol}_{d-1}(V(D,\x)),
\]
so
\begin{align*}
\lambda_n(D,\x) &\le c|\det T| \prod_{j=1}^{d}\rho_n(y_j) \le c(d,r,R,\nu,\sigma) {\rm Vol}_{d-1}(V(D,\x)) \delta^{1-d} \left(n^{-1}\sqrt{\delta}\right)^d \\
 &=c(d,r,R,\nu,\sigma)n^{-d}\delta^{1-d/2}{\rm Vol}_{d-1}(V(D,\x)),
\end{align*}
which concludes the proof.
\end{proof}

\section{Sharpness}\label{sect_proofs_sharp}

\begin{proof}[Proof of Theorem~\ref{upper_thm}.]
	
	First we treat the case $l_1l_2\le\delta$. For this case, it is sufficient to assume that $10\delta<l_1,l_2<\frac54$, which clearly implies $\delta<1$.
	
	Let $x$ and $y$ be Cartesian coordinates in $\R^2$. Set $Q=(2-\delta,0)$. Assuming that $l_2\ge l_1$, for $0\le\alpha\le 1$, let $Q_1$ and $Q_2$ be the two points of intersection of the line $x=-\alpha y+2-\delta$ with the circle $x^2+y^2=4$. Suppose that $-m_1<0<m_2$ are the $y$ coordinates of $Q_1$ and $Q_2$. It is straightforward that
	\[
	m_{1,2}=\frac{\sqrt{4\alpha^2+4\delta-\delta^2} \mp \alpha(2-\delta)}{1+\alpha^2},
	\]
	$m_1m_2=\delta(4-\delta)$, and using $0<\delta<1$ and $0\le \alpha\le 1$, we estimate that
	\[
	\alpha+\sqrt{\delta} \le \sqrt{4\alpha^2+4\delta-\delta^2} + \alpha(2-\delta) \le 4(\alpha+\sqrt{\delta}),
	\]
	which leads to
	\[
	\frac34\le \frac{m_1}{\frac{\delta}{\alpha+\sqrt{\delta}}}\le 4 \quad\text{and}\quad
	\frac12\le \frac{m_2}{\alpha+\sqrt{\delta}}\le 4.
	\]
	We want to choose $\alpha$ so that $l_2/l_1=m_2/m_1$. Note that if $\alpha=0$ then $m_2/m_1=1$, and that the quotient $m_2/m_1$ depends on $\alpha$ continuously. We have $l_2/l_1\ge1$ and in the other direction $l_2/l_1 < (5/4)/(10\delta)=1/(8\delta)$. If $\alpha=1$, then $m_2/m_1\ge1/(8\delta)$ which is bigger than $l_2/l_1$. Therefore, by continuity, the required choice of $\alpha\in[0,1]$ is possible.
	
	Consider the affine transform
	\[
	T(x,y)=(x+\alpha y,\mu y),
	\quad\text{where}\quad
	\mu=\frac{l_1}{m_1}=\frac{l_2}{m_2}=\sqrt{\frac{l_1l_2}{m_1m_2}}=\sqrt{\frac{l_1l_2}{\delta(4-\delta)}}.
	\]
	Such $T$ leaves any point on the $x$ axis unchanged (in particular, $T(Q)=Q$ and $T(2,0)=(2,0)$), satisfies $|T(Q_i)-Q|=l_i$, $i=1,2$, and that the line joining $T(Q_1)$ and $T(Q_2)$ is vertical. In other words, $\delta(T(2B^2),\x)=\delta$ and $l_i(T(2B^2),\x)=l_i$, $i=1,2$. Further, by~\eqref{ball}, $\lambda_n(2B^2,Q)\approx c(\sigma) n^{-2}\sqrt{\delta}$, hence, as $|\det T|=\mu\approx \sqrt{\frac{l_1l_2}{\delta}}$, due to~\eqref{affine}, we have $\lambda_n(T(2B^2),Q)\approx c(\sigma)  n^{-2}\sqrt{l_1 l_2}$. As $\mu\le1/\sqrt{3}$ (recall that we assume $l_1l_2\le \delta$) and $0\le\alpha\le 1$, one can see that $T(2B^2)\subset 4B^2$. So, we could take $T(2B^2)$ as the required $D$, but it may not be true that $B^2\subset T(2B^2)$. To achieve that, we let $D$ be the closed convex hull of $T(B^2)$ and $B^2$ (Christoffel function will not decrease due to~\eqref{compare}). We would like to verify that our measurements $\delta$, $l_1$ and $l_2$ do not change, i.e., that all three points $T(2,0)$, $T(Q_1)$, $T(Q_2)$ (which clearly belong to the boundary of $T(2B^2)$) are on the boundary of $D$. Note that the half-planes defined by the supporting lines to $2B^2$ at $(2,0)$ and $Q_i$, $i=1,2$, and containing $2B^2$ are given by the inequalities
	\begin{equation}\label{half-planes}
	x\le 2 \quad\text{and}\quad
	x\sqrt{4-m_i^2}\mp ym_i \le 4, \quad i=1,2,
	\end{equation}
	respectively. Therefore, it is enough to show that if $T(x,y)\in B^2$, then $x$ and $y$ satisfy~\eqref{half-planes}. If $T(x,y)\in B^2$, then $(x+\alpha y)^2+\mu^2y^2\le 1$, in particular, $x\le|\alpha y|+1$ and $|y|\le\frac1{\mu}$. We have (recall that $l_i>\delta/10$, $i=1,2$)
	\begin{equation}\label{aux1}
	|ym_i|\le \frac{m_i}{\mu} = \frac{m_1m_2}{l_{3-i}} \le \frac{4\delta}{l_{3-i}} \le \frac{2}{5}, \quad i=1,2,
	\end{equation}
	and
	\begin{equation}\label{aux2}
	|\alpha y|\le \frac{\alpha}{\mu}<\frac{\alpha+\sqrt{\delta}}{\mu}\le \frac{2m_2}{\mu} \le \frac{4}{5}.
	\end{equation}
	Now~\eqref{half-planes} readily follow from $x\le |\alpha y|+1$, \eqref{aux1} and~\eqref{aux2}. We verified that $D$ satisfies all the required conditions (also note that $\lambda_n(D,\x)\ge\lambda_n(T(2B^2),\x)$), which completes the case $l_1l_2\le\delta$.

	For the case $l_1l_2>\delta$, we first construct the required body $\overline D$ as in the case $l_1l_2\le\delta$ but for parameters $(\bar \delta, \bar l_1, \bar l_2)=(\delta, l_1, \delta/l_1)$. It is straightforward that $10\delta<l_1,l_2<\frac1{10}$ implies $10\bar\delta<\bar l_1,\bar l_2<\frac54$. Note that the line joining $T(Q_1)$ and $T(Q_2)$ is parallel to the supporting line to $\overline D$ at $T(1,0)$ (both are vertical). Choosing point $Q_3$ on the ray from $T(Q)$ to $T(Q_2)$ on the distance $l_2>\delta/l_1$ from $T(Q)$ and taking $D$ as the closed convex hull of $Q_3$ and $\overline D$ completes the proof.
\end{proof}

\begin{proof}[Proof of Theorem~\ref{mult_sharp}.]
We will choose large enough $\beta_1>0$ and small enough $\beta_2>0$ satisfying $(\beta_2/\beta_1)^{\frac1{d-1}}<1/2$ (which guarantees $\delta<1/2$) and certain additional conditions later in the proof.

Take $\x=(2-\delta,0,\dots,0)$, where $0<\delta<1/2$. Note that $2B^d$ restricted to the hyperplane $x_1=2-\delta$ is $\sqrt{\delta(4-\delta)}B^{d-1}$, so if we let
\begin{equation}\label{mu_sel}
\mu=\frac{v^{\frac1{d-1}}}{\sqrt{\delta(4-\delta)}({\rm Vol}_{d-1}(B^{d-1}))^{\frac1{d-1}}},
\quad\text{then}\quad
{\rm Vol}_{d-1}\left(\mu \sqrt{\delta(4-\delta)}B_{d-1},\x)\right)=v.
\end{equation}

First we consider the case $\mu\le 1$. Then $v^{\frac1{d-1}}\le c(d) \sqrt{\delta}$, and $\min\{\delta^{1-d/2}v,\sqrt{\delta}\} \approx c(d) \delta^{1-d/2}v$. Define
\[
T(x_1,\dots,x_d):=(x_1,\mu x_2, \dots, \mu x_d),
\]
then ${\rm Vol}_{d-1}(V(T(2B^d),\x))=v$. Also, as $x_1$ axis is unchanged under $T$, and we can take $\u=(1,0,\dots,0)$ to have that $\x+\delta\u\in\partial T(2B^d)$ and the hyperplane through $\x+\delta\u$ with normal vector $\u$ is supporting to $T(2B^d)$. By~\eqref{ball}, $\lambda_n(2B^d,\x)\approx c(d,\sigma) n^{-d} \sqrt{\delta}$, so, as $|\det T|=\mu^{d-1}\approx c(d) v \delta^{-(d-1)/2}$, due to~\eqref{affine}, we have $\lambda_n(T(2B^d),\x)\approx c(d,\sigma) n^{-d} \delta^{1-d/2}v$. Since $\mu\le 1$, we have $T(2B^d)\subset 2B^d$. We define the desired $D$ as the convex hull of the unit ball $B^d$ and $T(2B^d)$, then $B^d\subset D\subset 3B^d$, the properties of $\delta$ and $\u$ do not change, and by~\eqref{compare} we have $\lambda_n(D,\x)\ge \lambda_n(T(2B^d),\x)$. It remains to show that 
\begin{equation}\label{eqn-b1-ok}
\{\y\in T(2B^d):(\y-\x)\perp \u\}=\{\y\in D:(\y-\x)\perp \u\}.
\end{equation}
 It is enough to restrict our attention to two dimensions since $D$ and $T(2B^d)$ are invariant under rotation about $x_1$ axis. Straightforward computations show that the tangent line (in the first two dimensions) to $T(2B^d)$ at $(x_1,x_2)=(2-\delta,\mu\sqrt{\delta(4-\delta)})$ has the slope $\mu\frac{-2+\delta}{\sqrt{\delta(4-\delta)}}<0$ and intersects the vertical line $x_1=1$ at
\[
x_2=\mu\left(\frac{(2-\delta)(1-\delta)}{\sqrt{\delta(4-\delta)}} + \sqrt{\delta(4-\delta)}\right) \ge \frac {c \mu}{\sqrt{\delta}} \ge c(d) \beta_1^{\frac1{d-1}} \ge 1,
\]
where we choose $\beta_1$ to be sufficiently large. In other words, we established that this line is above the unit disc, which implies~\eqref{eqn-b1-ok}.

For the remaining case $\mu>1$, we have $v^{\frac1{d-1}}\ge c(d)\sqrt{\delta}$, so $\min\{\delta^{1-d/2}v,\sqrt{\delta}\} \approx c(d) \sqrt{\delta}$. Take $D$ to be the convex hull of $2B^d$ and the $(d-1)$-sphere
\[
\{(x_1,\dots,x_d):x_1=2-\delta,\,x_2^2+\dots+x_d^2=\mu^2\delta(4-\delta)\}.
\]
As in the previous case, we have $\dist(\x,\partial T(2B^d))=\delta$ and the choice $\u=(1,0,\dots,0)$ satisfies the required properties. Further, by~\eqref{mu_sel}, ${\rm Vol}_{d-1}(\{\y\in D:(\y-\x)\perp \u\})=v$, and 
\[
B^d\subset 2B^d\subset D\subset \sqrt{(2-\delta)^2+\mu^2\delta(4-\delta)}B^d\subset 3 B^d
\]
if we choose $\beta_2<{\rm Vol}_{d-1}(B^{d-1})$. Now~\eqref{compare} and~\eqref{ball} yield the required $\lambda_n(D,\x)\ge\lambda_n(2B^d,\x)\approx c(d,\sigma) \sqrt{\delta}$.
\end{proof}

{\bf Acknowledgement.} The author is grateful to Fend Dai for valuable comments.

\end{document}